\documentclass[12pt,letterpaper]{amsart}
\usepackage[margin=1.4in]{geometry}

\usepackage[T1]{fontenc}
\usepackage{lmodern}
\usepackage{microtype}
\usepackage{amsmath,amssymb,amsthm,mathtools}
\usepackage{xcolor}
\usepackage{enumitem}
\usepackage[colorlinks=true,linkcolor=blue!55!black,citecolor=blue!55!black,
            urlcolor=blue!55!black]{hyperref}
\usepackage[nameinlink, noabbrev]{cleveref}
\newtheorem{theorem}{Theorem}
\newtheorem{lemma}[theorem]{Lemma}
\newtheorem{proposition}[theorem]{Proposition}
\newtheorem{conjecture}[theorem]{Conjecture}
\newcommand{\kk}{\mathbb{R}}
\newcommand{\supp}{\operatorname{supp}}
\newcommand{\vsupp}{\operatorname{vsupp}}
\newcommand{\init}{\operatorname{in}_{\prec}}
\newcommand{\dist}{\operatorname{dist}}

\newcommand{\Nbd}{\operatorname{Nbd}}
\newcommand{\cS}{\mathcal{S}}
\newcommand{\cF}{\mathcal{F}}
\newcommand{\cH}{\mathcal{H}}
\newcommand{\cL}{\mathcal{L}}
\newcommand{\dk}{\partial_{k+1}}
\newcommand{\sk}{\mathcal{S}_k^a}
\newcommand{\skp}{\mathcal{S}_{k+1}^a}

\newcommand{\rk}{\operatorname{rk}}
\newlist{lemmaparts}{enumerate}{1}
\setlist[lemmaparts]{
  label=\textup{(\roman*)},
  ref=\textup{(\roman*)},
  leftmargin=*,
  topsep=.4\baselineskip,
  itemsep=.3\baselineskip,
  parsep=0pt
}

\crefname{lemmapartsi}{part}{parts}
\Crefname{lemmapartsi}{Part}{Parts}

\title{The $\boldsymbol{g}$-theorem in smooth approximation}
\author{Joel Hakavuori}
\date{}
\address{IMJ-PRG, Sorbonne Universit\'e, F-75005 Paris, France}
\email{hakavuori@imj-prg.fr}

\begin{document}
\begin{abstract}
We show that a sequence of simplicial polytopes approximating a smooth convex body has primitive Betti numbers diverging from the upper bound given by the $g$-theorem. This complements a result of Adiprasito--Nevo--Samper, who proved the corresponding divergence from the lower bound. The proof uses a new lower bound for the size of the shadow of a set of monomials.
\end{abstract}
\maketitle

\section{Introduction}

By the $g$-theorem \cite{Stanley, BilleraLee}, the $g$-vector of a simplicial $d$-polytope $P$ is the Hilbert function of a standard graded Artinian algebra. In other words, there exists a finite-dimensional graded quotient $A_* = \bigoplus_{i=0}^{\lfloor \frac d 2 \rfloor } A_i$ of a polynomial ring with $g_i = \dim A_i$. By Macaulay's theorem \cite{Macaulay}, the Hilbert function of a standard graded algebra satisfies $\dim A_{k} \geq \partial_{k+1} (\dim A_{k+1})$, where $\partial_i$ is the $i$th lower Macaulay function. This function is defined as the minimum possible cardinality of the shadow $\partial \cF := \{ m/x_j \ : \ m \in \cF, \ x_j \mid m\} $ of a set $\cF$ of degree-$i$ monomials with $|\cF| = a$. It can be defined explicitly as follows: for all positive integers $i$ and $a$, $a$ can be written uniquely in the form
\[
 a=\binom{a_i}{i}+\binom{a_{i-1}}{i-1}+\cdots+\binom{a_j}{j}
\]
with $a_i>a_{i-1}>\cdots>a_j\geq j\geq1$. Then, $\partial_i : \mathbb N \to \mathbb N$ is defined as
\[
 \partial_i (a)
 :=\binom{a_i-1}{i-1}+\binom{a_{i-1}-1}{i-2}
       +\cdots+\binom{a_j-1}{j-1},
\]
and $\partial_i (0) := 0$. Sequences of nonnegative integers beginning with 1 and satisfying Macaulay's inequalities are also known as $M$-vectors (or $O$-sequences) in the literature.

For an $M$-vector $(g_0, \ldots, g_r)$, we define the $k$th lower Macaulay defect as
\begin{equation*}
    \delta_k := g_k - \partial_{k+1}(g_{k+1})
\end{equation*}
for $1 \le k < r$. Thus $\delta_k \geq 0$ measures how far $g_k$ is from being minimal with respect to $g_{k+1}$. 

By taking initial ideals, numerical questions about Hilbert functions of quotients of standard graded polynomial rings are reduced to the case of monomial ideals, where the Hilbert function of the quotient can be computed by counting standard monomials. If $\mathcal O_i$ denotes the degree-$i$ standard monomials of such a quotient, so that $\partial \mathcal O_{k+1} \subseteq \mathcal O_{k}$, then
\[
 \partial_{k+1}(|\mathcal O_{k+1}|)
 \leq|\partial\mathcal O_{k+1}|
 \leq|\mathcal O_k|.
\]
Thus $\delta_k (S/I) = |\mathcal O_k| - \partial_{k+1}(|\mathcal O_{k+1}|) = 0$ implies that $|\partial\mathcal O_{k+1}|$ is minimal with respect to $|\mathcal O_{k+1}|$ and that there are no degree-$k$ standard monomials outside the shadow $\partial \mathcal O_{k+1}$. Thus, to have a small shadow, the degree-$k$ divisors of the degree-$(k+1)$ monomials must overlap significantly. For example, the set of the lexicographically smallest degree-$(k+1)$ monomials has a shadow of minimal cardinality. In the opposite direction, if the supports of the monomials are spread out in the variables, the divisors cannot overlap in the manner required for $\delta_k = 0$. We make this quantitative in \Cref{lem:separated-shadow}.


For simplicial polytopes, examples with $\delta_k(P) = g_k(P) - \partial_{k+1}(g_{k+1}(P)) = 0$ include all $(k+1)$-neighborly polytopes. In particular, every cyclic polytope has $\delta_k = 0$ throughout the whole $g$-vector. Thus $\delta_k = 0$ can occur even for polytopes with arbitrarily many faces of each dimension. The following conjecture of Kalai predicts that this extremality cannot hold for a sequence of simplicial polytopes approximating a smooth convex body.

\begin{conjecture}[Kalai {\cite{Kalai}}]
\label[conjecture]{conj:kalai}
Let $K\subseteq\mathbb{R}^d$ be a $C^1$ convex body, and let
$P_n$ be a sequence of simplicial $d$-polytopes converging to $K$ in
the Hausdorff metric. Then
\begin{enumerate}[label=\textup{(\roman*)}]
 \item for every $1\leq k\leq\left\lfloor\frac d2\right\rfloor$,
 \[
  g_k(P_n)\longrightarrow \infty,
 \]

 \item and for every $1\leq k\leq\left\lfloor\frac d2\right\rfloor-1$,
 \[
  \delta_k (P_n) = g_k(P_n)-\partial_{k+1}\bigl(g_{k+1}(P_n)\bigr)
  \longrightarrow \infty
 \]
 as $n \to \infty$.
\end{enumerate}
\end{conjecture}

The first part of this conjecture was proved by Adiprasito--Nevo--Samper \cite[Theorem~4.5]{ANS} by producing, for any $s>0$, $s$ linearly independent affine $k$-stresses on $P_n$ for large enough $n$. Additionally, using results of B\'ar\'any \cite{Barany}, they proved part~\textup{(ii)} for the model for random polytopes considered in \cite[Remark~5.7]{ANS}. The purpose of this paper is to prove part~\textup{(ii)} of \Cref{conj:kalai} in full.

Let $\Delta$ be a simplicial complex. For a nonempty subset $A\subseteq V(\Delta)$ and an integer $r\geq0$, define $\Nbd_\Delta^r[A]:= \{v\in V(\Delta):\dist_{\Delta^{(1)}}(v,A)\leq r\}$, where $\dist$ denotes graph distance in the 1-skeleton $\Delta^{(1)}$. For a subcomplex $\Gamma \subseteq \Delta$, set $\Nbd_\Delta^r[\Gamma] := \Nbd_\Delta^r[V(\Gamma)]$.

Let $\gamma \subseteq \partial P = \Delta$ be a simplicial $k$-sphere that is a subcomplex of $\Delta$. We say that $\gamma$ is \emph{homologically full} in $\Delta$ if 
\[
 [\gamma] \neq 0\in
 \widetilde H_{k}
 \bigl(\Delta_{V(\gamma)};\mathbb R\bigr),
\]
where $[\gamma]$ is the fundamental class of $\gamma$ and $\Delta_{V(\gamma)}$ is the induced subcomplex of $\Delta$ on the vertices of $\gamma$. For example, every simplicial sphere $\gamma$ that is an induced subcomplex of $\Delta$ is automatically homologically full.

\begin{theorem}
\label[theorem]{mainthm}
    Let $1 \le k \le \lfloor \frac{d}{2}\rfloor - 1$, and let $P \subseteq \mathbb R^d$ be a simplicial $d$-polytope with $\Delta = \partial P$. Assume there exist $s$ homologically full $(k-1)$-spheres $\gamma_1, \ldots, \gamma_s$ in $\partial P$, with
    \begin{equation*}
        \Nbd_\Delta^2[\gamma_i] \cap \Nbd_\Delta^2[\gamma_j] = \varnothing
    \end{equation*}
    for $i \neq j$. Then
    \begin{equation*}
        g_k(P) - \partial_{k+1}(g_{k+1}(P)) \geq s - \partial_{k+1}(s).
    \end{equation*}
\end{theorem}

We now give an outline of the proof of \Cref{mainthm}. Let $\mathcal S^a_k \subseteq \mathbb R [x_v \ : \ v \in V(P)]_k$ denote the space of affine $k$-stresses on $\partial P$. These are homogeneous degree-$k$ polynomials that are supported on faces of $\partial P$, and satisfy certain constant-coefficient linear differential equations. Affine stresses are the Macaulay inverse system of the quotient of the Stanley--Reisner ring of $\partial P$ by the geometric linear system of parameters given by the vertex coordinates of $P$ and the Lefschetz element $\sum_v x_v$. Fundamental properties of affine stresses include that they form a vector space of dimension $\dim_\mathbb{R} \cS^a_k = g_k(P)$, and that a linear constant-coefficient derivative of a $(k+1)$-stress is a $k$-stress. Let $D_k \subseteq \cS^a_k$ denote the vector space spanned by such derivatives of affine $(k+1)$-stresses.

In this setting, Macaulay's theorem can be stated as $\dim D_k \geq \partial_{k+1}(\dim \cS^a_{k+1})$, so we have $g_k = \dim \cS^a_{k} \geq \dim D_k \geq \partial_{k+1}(\dim \cS^a_{k+1})$. We thus write the $k$th defect $\delta_k$ as
\[
 \begin{aligned}
 \delta_k = g_k(P)-\partial_{k+1}\bigl(g_{k+1}(P)\bigr)
 &=
 \bigl [ \dim \cS^a_{k} -\dim D_{k}\bigr ]\\
 &\quad+
 \bigl [\dim D_{k}-\partial_{k+1}(\dim \cS^a_{k+1})\bigr ].
 \end{aligned}
\]
By \Cref{lem:local-stress}, the $s$ homologically full $(k-1)$-spheres $\gamma_1, \ldots, \gamma_s$ give us linearly independent affine $k$-stresses $\lambda_1, \ldots, \lambda_s$ in $\cS^a_{k}$. Let $L$ denote the span of these stresses, and let $t := \dim (L \cap D_k)$. The quotient $L/(L\cap D_k)$ injects into $\cS_k^a /D_k$, giving a lower bound of $s-t$ for the first term. 

For the second term, let $\cF$ denote the initial monomials of $\skp$ and $\cH$ the initial monomials of $D_k$ with respect to some monomial order. Then $|\cF| = \dim \skp$, $|\cH| = \dim D_k$, and $\partial \cF \subseteq \cH$. We then further split the second term into
\[
    \dim D_{k}-\partial_{k+1}(\dim \cS^a_{k+1}) = \bigl [ |\cH| - |\partial \cF|\bigr ] + \bigl [ |\partial \cF| - \partial_{k+1}(|\cF |) \bigr ] .
\]
Let $\mu_1, \ldots, \mu_s$ be the initial monomials of $\lambda_1, \ldots, \lambda_s$. We may assume that $\{\mu_1, \ldots, \mu_t\}$ belong to $\cH = \init D_k$. Assume that $\{\mu_1, \ldots, \mu_p\} \subseteq \partial \cF$, and $\{\mu_{p+1}, \ldots, \mu_t\} \subseteq \cH \setminus \partial \cF$. We then get a lower bound of $t - p$ for the first term in this second split, and the main combinatorial result of the paper, \Cref{lem:separated-shadow}, gives a lower bound of $p - \partial_{k+1}(p)$ for $|\partial \cF| - \partial_{k+1}(|\cF |)$. When combined, these three lower bounds prove the main result.

To deduce part~\textup{(ii)} of \Cref{conj:kalai} from \Cref{mainthm}, we use the construction of Adiprasito--Nevo--Samper, reviewed in \Cref{lem:distant-spheres}, that supplies us with increasingly many homologically full $(k-1)$-spheres $\{\gamma_i \}_{i=1}^s$ as $P_n \to K$. These spheres satisfy the hypothesis in \Cref{mainthm}. As $\partial_{k+1}(s) = o(s)$, we get 
\[ 
    \delta_k \geq s - \partial_{k+1}(s) \longrightarrow \infty.
\] 
Finally, in \Cref{sec:matroid-chow} we give a further application of \Cref{lem:separated-shadow} to matroid Chow rings. We show that a collection of hyperplanes with pairwise intersections of rank at most one forces a Macaulay defect in the Hilbert function of the Chow ring.

\section{Shadows and initial subspaces}
\label[section]{sec:shadows}
The goal of this section is to prove \Cref{lem:separated-shadow}. We start with definitions and notation.

Let $\Delta$ be a simplicial complex on a finite set $V = V(\Delta)$, and let $R:=\mathbb{R}[x_v:v\in V]$ be the polynomial ring over the variables corresponding to vertices of $\Delta$. We say that a monomial $m = x^\alpha \in R$ is a \emph{$\Delta$-monomial} if $\supp (m) = \{v \ : \alpha_v > 0 \} \in\Delta$. We write $R_k(\Delta)$ for the vector subspace of $R_k$ spanned by the degree-$k$ $\Delta$-monomials.

For a finite collection $\cF$ of degree-$k$ monomials, its lower shadow is defined as $\partial\cF:=\{m/x_v\ :\ m\in\cF,\ x_v\mid m\}$. For a nonempty subset $A\subseteq V(\Delta)$ and an integer $r\geq0$, recall that $\Nbd_\Delta^r[A]:= \{v\in V(\Delta):\dist_{\Delta^{(1)}}(v,A)\leq r\}$,
where $\dist_{\Delta^{(1)}}(A,B) := \min\{\dist_{\Delta^{(1)}}(a,b):a\in A,\ b\in B\}$ denotes graph distance in the 1-skeleton $\Delta^{(1)}$. We abbreviate $\Nbd_\Delta[A]:=\Nbd_\Delta^1[A]$. For a polynomial $f \in R$, let $\vsupp (f) := \cup_{m \in \supp (f)} \supp (m) \subseteq V(\Delta)$ denote the set of vertices on which $f$ is supported.

We start with the following numerical properties of the $k$th lower Macaulay function $\partial_k$ that will be needed in the proof of \Cref{lem:separated-shadow}.

\begin{lemma}
\label[lemma]{lem:macaulay-numerics}
\leavevmode\par
\begin{lemmaparts}
\item
\label{part:macaulay-balance}
Let $k, p, b$ be integers with $k\geq2$, $p\geq 1$, and $0\leq b<\partial_k(p)$. Then
\begin{equation*}
 \partial_{k-1}(b)+\partial_k(p-b)\geq \partial_k(p).
\end{equation*}

\item 
\label{part:macaulay-sum}
Let $k, p, b_1, \ldots, b_p$ be integers with $k\geq1$, $p\geq1$, and $b_i \geq 1$ for all $i$. Then
\begin{equation*}
 \sum_{i=1}^p \partial_k(b_i)
 \geq
 \partial_k\!\left(\sum_{i=1}^p b_i-p+1\right)+p-1.
\end{equation*}

\end{lemmaparts}
\end{lemma}

\begin{proof}
\Cref{part:macaulay-balance} is \cite[Theorem~1]{AFML}. For \cref{part:macaulay-sum}, the case $p=1$ is an equality, so consider the case $p=2$. Let $b,c>0$. Set $\mathcal B$ and $\mathcal C$ to be the lexicographically smallest $b$ and $c$ degree-$k$ monomials in two sets of variables intersecting only in $z$, with $z$ the smallest variable in both orders. Their shadows have cardinalities $\partial_k(b)$ and $\partial_k(c)$, respectively. 

Then $\mathcal B\cap\mathcal C=\{z^k\}$, $\partial\mathcal B\cap\partial\mathcal C=\{z^{k-1}\}$, and $|\mathcal B\cup\mathcal C|=b+c-1$. Furthermore, we have $|\partial(\mathcal B\cup\mathcal C)| = \partial_k(b)+\partial_k(c)-1$. As $\partial_k$ is the minimal size of the shadow of a set of degree-$k$ monomials, we have
\begin{equation}
 \partial_k(b+c-1)
 \leq
 \partial_k(b)+\partial_k(c)-1.
 \label{eq:macaulay-pair}
\end{equation}
The $p\ge 3$ case follows by induction: set $c:=\sum_{i=1}^{p-1}b_i-p+2$, so that inductively we have
\[ \partial_k(c)
 \leq
 \sum_{i=1}^{p-1}\partial_k(b_i)-(p-2).
\]
Applying \cref{eq:macaulay-pair} to $c$ and $b_p$, we obtain
\[
 \begin{aligned}
 \partial_k\!\left(\sum_{i=1}^p b_i-p+1\right)
 &=
 \partial_k(c+b_p-1)\\
 &\leq
 \partial_k(c)+\partial_k(b_p)-1\\
 &\leq
 \sum_{i=1}^p \partial_k(b_i)-(p-1),
 \end{aligned}
\]
proving part (ii).
\end{proof}

Next, we prove the main combinatorial result of the paper. Roughly speaking, the lemma states that the more spread out a set $\cF$ of $\Delta$-monomials is on $\Delta$, the further away $|\partial \cF|$ is from being minimal with respect to $|\cF|$.

\begin{lemma}
\label[lemma]{lem:separated-shadow}
Let $\Delta$ be a simplicial complex, and let $\cF$ be a finite set of degree-$k$ $\Delta$-monomials for $k\geq2$. Suppose that there are $v_1,\ldots,v_p\in V(\Delta)$ satisfying
\[
 \Nbd_\Delta[v_i]\cap\Nbd_\Delta[v_j]=\varnothing
\]
for $i \neq j$, and that for every $v_i$, some monomial of $\cF$ is divisible by $x_{v_i}$. Then
\begin{equation}
 |\partial\cF|-\partial_k(|\cF|)
 \geq
 p-\partial_k(p).
 \label{eq:separated-shadow-defect}
\end{equation}
\end{lemma}

\begin{proof}
The $p=0$ case is Macaulay's theorem, so assume that $p\geq1$. As no face of $\Delta$ contains any two $v_i$, no monomial is divisible by two of the $v_i$, so $|\cF|\geq p$. Set $N:=|\cF|-p+1$. We prove the lower bound
\begin{equation*}
 |\partial\cF|
 \geq
 \partial_k(N)+p-1
\end{equation*}
for the shadow. Once this is known, \cref{part:macaulay-sum} of \Cref{lem:macaulay-numerics} gives
\[
 \partial_k(|\cF|)
 =
 \partial_k(N+p-1)
 \leq
 \partial_k(N)+\partial_k(p)-1.
\]
Combining these bounds gives
\begin{align*}
    |\partial\cF|-\partial_k(|\cF|) & \geq \partial_k(N)+p-1 - (\partial_k(N)+\partial_k(p)-1) = p - \partial_k (p).
\end{align*}
We now prove $|\partial\cF| \geq \partial_k(N)+p-1$. For $1\leq i\leq p$, define $\cF_i:=\{m\in\cF:x_{v_i}\mid m\}$ and $\cL_i:=\{m/x_{v_i}:m\in\cF_i\}$, and let $\cF_0 := \cF\setminus\bigcup_{i=1}^p\cF_i$.
The assumption $\Nbd_\Delta[v_i]\cap\Nbd_\Delta[v_j]=\varnothing$ implies that we have a disjoint decomposition $\cF = \bigsqcup_{i=0}^p \cF_i$.
Write $a_i:=|\cF_i|$ for all $0 \le i \le p$, so that $a_i=|\cF_i|=|\cL_i|$ for all $i \geq 1$. Set $A:=\sum_{i=1}^p a_i$, so that $|\cF|=a_0+A$. 

Next, we obtain two lower bounds for $|\partial\cF|$. First, $\cL_i\subseteq\partial\cF$ for every $i$, and the assumption $\Nbd_\Delta[v_i]\cap\Nbd_\Delta[v_j]=\varnothing$ implies that $\cL_1,\ldots,\cL_p$ are pairwise disjoint. Every monomial in $\cL_i$ has degree $k-1\geq 1$, and thus
\begin{equation*}
 |\partial\cF|\geq A.
\end{equation*}

Second, the sets $\partial\cF_0, x_{v_1}\partial\cL_1,\ldots, x_{v_p}\partial\cL_p$ are pairwise disjoint subsets of $\partial\cF$: no monomial in $\partial \cF_0$ is divisible by any $x_{v_i}$, and the sets $x_{v_1}\partial\cL_1,\ldots, x_{v_p}\partial\cL_p$ are pairwise disjoint by the assumption on the neighborhoods of $v_i$. Macaulay's theorem then gives
\begin{align*}
 |\partial\cF|
 &\geq
 |\partial\cF_0|
 +
 \sum_{i=1}^p|\partial\cL_i|\\
 &\geq
 \partial_k(a_0)
 +
 \sum_{i=1}^p\partial_{k-1}(a_i).
\end{align*}

Set $a:=A-p+1$, so that $N=a_0+a$. By \cref{part:macaulay-sum} of
\Cref{lem:macaulay-numerics}, we have
\[
 \sum_{i=1}^p\partial_{k-1}(a_i)
 \geq
 \partial_{k-1}(a)+p-1.
\]
These two lower bounds for $|\partial \cF|$ give
\begin{equation}
 |\partial\cF|
 \geq
 \max\bigl\{
   a,\,
   \partial_{k-1}(a)+\partial_k(a_0)
 \bigr\} + p-1.
 \label{eq:shadow-max-bound}
\end{equation}
Now, observe that the maximum in \cref{eq:shadow-max-bound} is at least $\partial_k(N)$: if $a\geq \partial_k(N)$, this is immediate, and if $a<\partial_k(N)$, then \cref{part:macaulay-balance} of \Cref{lem:macaulay-numerics} gives
\[
 \partial_{k-1}(a)+\partial_k(N-a)
 \geq
 \partial_k(N).
\]
Since $N-a=a_0$, the second expression inside the maximum in
\cref{eq:shadow-max-bound} is at least $\partial_k(N)$. Therefore
\[
 |\partial\cF|
 \geq
 \partial_k(N)+p-1
\]
as desired.
\end{proof}

\section{Initial subspaces}
\label{sec:initspaces}
To apply \Cref{lem:separated-shadow} to affine stresses, we will take initial subspaces of $\sk$. Let $\prec$ be a monomial order on $R:= \mathbb R [x_v \ : \ v\in V]$ for a finite set $V$. For a homogeneous degree-$k$ polynomial $f \in R_k$, let $\operatorname{in}_{\prec}(f)$ denote the initial monomial of $f$, with coefficient 1. For a vector subspace $W \subseteq R_k$, let $\operatorname{in}_{\prec}(W)$ denote the set of initial monomials of polynomials in $W$. 

For a vector subspace $W \subseteq R_k$, let
\[
    D(W):=\sum_{v\in V}\partial_v W\subseteq R_{k-1}
\]
denote the span of the linear coordinate derivatives of polynomials in $W$. Here $\partial_v:=\partial/\partial x_v$.

The following result is standard, see, for example, \cite[Chapter~15]{Eisenbud}.
\begin{lemma}
\label[lemma]{lem:init}
    Let $W \subseteq R_k$ be a vector subspace for $k \geq 1$. Set $\cF := \operatorname{in}_{\prec}(W)$ and $\cH := \operatorname{in}_{\prec}(D(W))$.
    Then 
    \begin{align*}
        |\cF| & = \dim W, \\ 
        |\cH| & = \dim D(W),
    \end{align*}
    and
    \[
        \partial \cF \subseteq \cH.
    \]
\end{lemma}
The inclusion $\partial \cF \subseteq \cH$ follows from the observation that $\init(\partial_v f)=\init(f)/x_v$ whenever $x_v \mid \init (f)$, which in turn follows from the multiplicativity of the monomial order and the ground field having characteristic zero. Note that the property of a polynomial having support on a simplicial complex $\Delta$ is preserved under taking initial terms. 

The following lemma combines the previous two, and is the main lower bound used in the proof of \Cref{mainthm}.

\begin{lemma}
\label[lemma]{mainbound}
    Let $W \subseteq R_k(\Delta)$ be a vector subspace, with $k\geq 2$. Let $f_1, \ldots, f_t$ be nonzero polynomials in $D(W)\subseteq R_{k-1}(\Delta)$, and let $\mu_i := \init (f_i)$ for $1 \le i \le t$. Assume that
    \begin{equation*}
        \Nbd_\Delta [\supp (\mu_i)] \cap \Nbd_\Delta [\supp(\mu_j)] = \varnothing
    \end{equation*}
    for $i \neq j$. Then
    \begin{equation*}
        \dim D(W) - \partial_{k} (\dim W) \geq t - \partial_k (t). 
    \end{equation*}
\end{lemma}
\begin{proof}
    Set $\cF:= \init (W)$ and $ \cH:=\init \bigl ( D(W)\bigr)$. By \Cref{lem:init}, we have $|\cF|=\dim W$, $|\cH|=\dim D(W)$ and $\partial\cF\subseteq\cH$, so that $\dim D(W) \geq |\partial\cF| \geq \partial_k (|\cF|) $. Thus
    \begin{equation}
    \label{secondsplit}
        \dim D(W) - \partial_{k} (\dim W) = \bigl [|\cH| - |\partial \cF| \bigr ] + \bigl [|\partial \cF| - \partial_{k} (|\cF|)\bigr ].
    \end{equation}

    As $f_i \in D(W)$, we have $\mu_i \in \cH$ for all $i$. The monomials $\mu_i$ have positive degree and are pairwise distinct.
    After reordering, we may assume that $\{\mu_1, \ldots, \mu_p\} \subseteq \partial \cF$, and $\{\mu_{p+1}, \ldots, \mu_t\} \subseteq \cH \setminus \partial \cF$.  For each $\mu_1, \ldots, \mu_p$, choose a vertex $v_i$ with $x_{v_i} \mid \mu_i$. Since $\mu_i\in\partial\cF$, it divides some monomial of $\cF$, which is consequently divisible by $x_{v_i}$. Moreover,
    \[
     \Nbd_\Delta[v_i]
     \subseteq\Nbd_\Delta[\supp(\mu_i)].
    \]
    By the neighborhood assumption, we may apply \Cref{lem:separated-shadow} with these $v_1,\ldots, v_p$ to get the lower bound
    \[
        |\partial \cF| - \partial_k (|\cF|) \geq p - \partial_{k} (p)
    \]
    for the second term in \cref{secondsplit}. As $\{\mu_{p+1}, \ldots, \mu_t\} \subseteq \cH \setminus \partial \cF$, we have $|\cH| - |\partial \cF| \geq t - p$, so that
    \begin{align*}
        \dim D(W) - \partial_{k} (\dim W) & \geq t - p + p - \partial_{k} (p) \\
        & \geq t - \partial_k (t),
    \end{align*}
    where we used $p\leq t$ and the monotonicity of $\partial_k$.
\end{proof}

\section{Local stresses and separated spheres}
\label{sec:distant}
In this section, we present the results of Adiprasito--Nevo--Samper required for the proofs of the main results. \Cref{lem:local-stress} produces an affine $k$-stress for each homologically full $(k-1)$-sphere in $\partial P$, which is used for the proof of \Cref{mainthm}. \Cref{lem:distant-spheres} constructs increasingly many separated homologically full spheres as $P_n \to K$, which combined with \Cref{mainthm} proves part~\textup{(ii)} of \Cref{conj:kalai}.

We first recall the inverse system realization of affine stresses.  Let
$P\subseteq\mathbb{R}^d$ be a simplicial $d$-polytope with boundary
complex $\Delta$.  Translate $P$ so that the origin lies in its
interior, write $p(v)=(p_1(v),\ldots,p_d(v))$ for its vertices, and
consider the Stanley--Reisner ring $\kk[\Delta] =\kk[X_v:v\in V(\Delta)]/I_\Delta$. We use the linear system of parameters given by $\theta_j:=\sum_v p_j(v)X_v$ for $1\leq j\leq d$, and the Lefschetz element $\ell = \sum_vX_v$. In the dual variables, the space of affine $k$-stresses is
\begin{equation*}
 \cS_k^a(P)
 :=\left\{
 f\in R_k(\Delta):
 \left(\sum_v p_j(v)\partial_v\right)f=0\ \text{for} \ 1\leq j\leq d,
 \quad
 \left(\sum_v\partial_v\right)f=0
 \right\}.
\end{equation*}
This is the degree-$k$ inverse system of $\kk[\Delta]/(\theta_1,\ldots,\theta_d,\ell)$. This forms a vector space of dimension
\[
 \dim\cS_k^a(P)=g_k(P)
\]
for $0\leq k\leq \lfloor\frac{d}{2} \rfloor$ (for background on stress spaces, see, for example, \cite{Lee,McMullen}, \cite[Theorem~4.7]{ANS}, or \cite[Section~2]{MNZ}). Affine stresses can equivalently be defined using Minkowski weights. In this description, the coefficients of squarefree terms, corresponding to simplices in $\partial P$, satisfy the Minkowski balancing equations, and these squarefree coefficients determine the corresponding polynomial uniquely. 


As constant-coefficient differential operators commute and face support is closed under taking coordinate derivatives, we additionally have
\[
 D\bigl(\cS_{k+1}^a(P)\bigr)\subseteq\cS_k^a(P).
\]

Recall that a simplicial $(k-1)$-sphere $\gamma\subseteq\Delta$ is homologically full in $\Delta$ if the fundamental class of $\gamma$ maps to a nonzero homology class in $\widetilde H_{k-1} \bigl(\Delta_{V(\gamma)};\mathbb R\bigr)$ under the inclusion $\gamma \hookrightarrow \Delta_{V(\gamma)}$. In particular, this condition holds when $\gamma$ is an induced subcomplex. 

The following lemma gives the connection between the homologically full $(k-1)$-spheres in $\partial P$ appearing in \Cref{mainthm} and $g_k(P)$.

\begin{lemma}[{\cite[Lemma~4.6]{ANS}}]
\label[lemma]{lem:local-stress}
Let $P\subseteq\mathbb R^d$ be a simplicial $d$-polytope with $\Delta=\partial P$ and $1\leq k\leq\left\lfloor\frac d2\right\rfloor$. Suppose that $\gamma\subseteq\Delta$ is a homologically full simplicial $(k-1)$-sphere in $\Delta$. Then there exists a nonzero affine $k$-stress
$\lambda\in\cS_k^a(P)$ with
\[
 \vsupp(\lambda)
 \subseteq\Nbd_\Delta[V(\gamma)].
\]
\end{lemma}

Next, we review the geometric construction of Adiprasito--Nevo--Samper that supplies us with increasingly many homologically full spheres in $\partial P_n$ as $n \to \infty$, all satisfying the hypothesis of \Cref{lem:local-stress}. The additional property we require is that for large enough $n$, these spheres can be separated by any prescribed graph distance $R$. The $R=1$ case is proved in \cite[Theorem~4.5]{ANS}, and the proof of the general case is essentially the same. We include the main ideas for completeness.

\begin{lemma}[{\cite[Theorem~4.5]{ANS}}]
\label[lemma]{lem:distant-spheres}
Let $K\subseteq\mathbb R^d$ be a $C^1$ convex body, and let
$P_n\to K$ be a sequence of simplicial $d$-polytopes converging
in the Hausdorff metric. Let $\Delta_n:=\partial P_n$ and $1\leq k\leq\left\lfloor\frac d2\right\rfloor-1$. For all positive integers $s$ and $R$, and all sufficiently large $n$, there exist $s$ homologically full simplicial $(k-1)$-spheres $\gamma_{1,n},\ldots,\gamma_{s,n}\subseteq\Delta_n$ such that
\[
 \dist_{\Delta_n^{(1)}}(V(\gamma_{i,n}),V(\gamma_{j,n}))>R
\]
for $i \neq j$.
\end{lemma}

\begin{proof}
Using the geometric construction in the proof of \cite[Theorem~4.5]{ANS} with sufficiently generic projections $\pi_i$ for $\ 1 \le i \le s$ onto $k$-dimensional subspaces, we get $s$ pairwise disjoint spheres $\{\eta_i\}_{i = 1}^s \subseteq\partial K$ defined as
\[
 \eta_i := K \cap \pi_i^{-1}(\partial \pi_i (K)).
\]
These spheres satisfy
\[
 K\cap H_y\subseteq \eta_i
\]
for all $y\in\eta_i$, where $H_y$ is the supporting hyperplane of $K$ at $y$. As $K$ is $C^1$, this hyperplane is unique.

Now, \cite[Lemmas~4.1--4.4]{ANS} imply that, for each $\eta_i$ and all large enough $n$, there exists a simplicial $(k-1)$-sphere $\gamma_{i,n}\subseteq\Delta_n$ close to $\eta_i$ in the Hausdorff metric that is homologically full in $\Delta_n$.

It remains to prove the assertion on the pairwise distance between the vertex sets $V(\gamma_{i,n})$. For a contradiction, suppose that this fails for infinitely many $n$. After passing to a subsequence, there are indices $i\neq j$, a fixed integer $1\leq m\leq R$, and paths $v_{0,n},\ldots,v_{m,n}$ of adjacent vertices with $v_{0,n}\in V(\gamma_{i,n})$ and $v_{m,n}\in V(\gamma_{j,n})$. Furthermore, $v_{r,n}\longrightarrow v_r$ for $0\leq r\leq m$, $v_0\in\eta_i$, and $v_m\in\eta_j$.

Now, for each $1\leq r\leq m$, let $H_{r,n}$ be a supporting hyperplane of $P_n$ containing the edge $[v_{r-1,n},v_{r,n}]$. Hausdorff convergence to $K$ then gives a supporting hyperplane $H_r$ of $K$ containing both $v_{r-1}$ and $v_r$. Since $K$ is $C^1$, the supporting hyperplane at each boundary point is unique, so $H_1=\cdots=H_m=H_{v_m}$. Thus $v_0\in K\cap H_{v_m}\subseteq\eta_j$, contradicting $v_0\in\eta_i$ and $\eta_i\cap\eta_j=\varnothing$. 

\end{proof}
 
\section{Proof of main theorem}
\begin{proof}[Proof of \Cref{mainthm}]
    Let $\sk$ and $\skp$ be the spaces of affine $k$- and $(k+1)$-stresses on $P$ respectively. Let $D_k:= D(\skp) \subseteq \sk$ be the span of the coordinate derivatives of $\skp$ in $\sk$. We then have $g_k = \dim \cS_k^a \geq \dim D_k \geq \partial_{k+1}(\dim \skp) = \partial_{k+1} (g_{k+1})$, where the second of the inequalities is Macaulay's theorem. We split the defect as
    \begin{equation}\label{deficitsplit}
        \delta_k = \bigl [ \dim \cS_k^a - \dim D_k\bigr ] +\bigl [ \dim D_k - \dk (\dim \cS_{k+1}^a) \bigr].
    \end{equation}

    By \Cref{lem:local-stress}, for each homologically full $(k-1)$-sphere $\gamma_i \subseteq \partial P, \ 1 \leq i \leq s$, we have a corresponding affine $k$-stress $\lambda_i \in \sk$ that satisfies $\vsupp (\lambda_i) \subseteq \Nbd_\Delta [\gamma_i]$, so that $\Nbd_\Delta[\vsupp(\lambda_i)] \subseteq \Nbd_\Delta^2[\gamma_i]$. The separation assumption
    \[
        \Nbd_\Delta^2[\gamma_i] \cap \Nbd_\Delta^2[\gamma_j] = \varnothing
    \]
    for $i \neq j$ implies the $\lambda_i$ are linearly independent. 
    
    Let $L$ denote the span of $\{\lambda_1, \ldots, \lambda_s\}$ in $\sk$. Set $t := \dim (L \cap D_k)$. We have an injection $L / (L \cap D_k) \hookrightarrow \sk / D_k$, so we get the lower bound 
    \[ 
        \dim \cS_k^a - \dim D_k \geq s - t
    \]
    for the first term in \cref{deficitsplit}.
    
    For the second term, let $\mu_i = \init (\lambda_i)$ for $1 \le i \le s$. As the $\lambda_i$ have disjoint supports, the initial monomial of every nonzero $f \in L$ is some $\mu_i$, so $\init (L \cap D_k) \subseteq \{\mu_1, \ldots, \mu_s\}$. By \Cref{lem:init}, we have $|\init (L \cap D_k)| = t$. The separation assumption on the spheres $\{\gamma_i\}_{i=1}^s$ and $\supp(\mu_i) \subseteq \vsupp(\lambda_i)$ then give
    \[
    \Nbd_\Delta [\supp (\mu_i)] \cap \Nbd_\Delta [\supp (\mu_j)] = \varnothing
    \]
    for $i \neq j$. Thus, we can apply \Cref{mainbound} with $W = \cS_{k+1}^a$ and $f_1,\ldots, f_t \in L \cap D_k$ polynomials whose initial monomials are the $t$ distinct elements of $\init (L \cap D_k)$ to get 
    \[ 
        \dim D_k - \dk (\dim \cS_{k+1}^a) \geq t - \partial_{k+1}(t),
    \]
    so that
    \begin{align*}
        \delta_k & \geq \bigl [ s - t \bigr ]+ \bigl [ t - \partial_{k+1}(t)\bigr ] \\
            & \geq s - \partial_{k+1}(s),
    \end{align*}
    where we used $t \le s$ and the monotonicity of $\partial_{k+1}$.
\end{proof}
In the proof, the homologically full spheres are used only to obtain the local stresses. The same proof therefore gives
\[
 g_k(P)-\partial_{k+1}\bigl(g_{k+1}(P)\bigr)
 \geq s-\partial_{k+1}(s)
\]
whenever there exist $s$ nonzero affine $k$-stresses $\lambda_1, \ldots, \lambda_s$ with $\Nbd_\Delta [\vsupp (\lambda_i)] \cap \Nbd_\Delta [\vsupp (\lambda_j)] = \varnothing$ for $i \neq j$.

\begin{proof}[Proof of part~\textup{(ii)} of \Cref{conj:kalai}]
As $P_n$ converges to $K$ in the Hausdorff distance, \Cref{lem:distant-spheres} with $R=4$ supplies us with increasingly many homologically full $(k-1)$-spheres $\{\gamma_i\}_{i=1}^s$, all satisfying the hypothesis of \Cref{mainthm}. Thus \Cref{mainthm}, combined with the fact that $\partial_{k+1}(s) = O(s^{\frac{k}{k+1}}) = o(s)$, gives us
\begin{equation*}
    \delta_k \geq s - \partial_{k+1}(s) \longrightarrow \infty.
\end{equation*}
\end{proof}

\section{Matroid Chow rings}
\label{sec:matroid-chow}

To conclude, we give a further application of \Cref{lem:separated-shadow} to Hilbert functions of matroid Chow rings. Let $M$ be a rank-$r$ loopless matroid with rank function $\rk$, and let $A^\bullet(M)$ be its Chow ring \cite{FeichtnerYuzvinsky}. Denote $a_i(M):=\dim_{\mathbb R}A^i(M)$.

We use the Feichtner--Yuzvinsky (FY) monomial basis \cite[Corollary~1]{FeichtnerYuzvinsky} of $A^\bullet(M)$.  For the lattice of flats $\mathcal L_M$ of $M$ with its maximal building set, this basis consists of the monomials
\[
    x_{F_1}^{b_1}\cdots x_{F_\ell}^{b_\ell}
\]
for all chains of flats $\varnothing=F_0\subsetneq F_1\subsetneq\cdots\subsetneq F_\ell$ with $0<b_i<\rk(F_i)-\rk(F_{i-1})$. A divisor of an FY-monomial is also an FY-monomial.

\begin{proposition}
\label[proposition]{prop:matroid-chow}
Let $M$ be a rank-$r$ loopless matroid, and suppose that $H_1,\ldots,H_p$ are hyperplanes of $M$ such that
\[
 \rk(H_i\cap H_j)\leq 1
\]
for all $i \neq j$. Then, for every $2\leq k\leq r-2$,
\[
 a_{k-1}(M)-\partial_k\bigl(a_k(M)\bigr)
 \geq
 p-\partial_k(p).
\]
\end{proposition}

\begin{proof}
Let $\cF_k$ denote the degree-$k$ Feichtner--Yuzvinsky monomials, so that $|\cF_k|=a_k(M)$ and $\partial\cF_k\subseteq\cF_{k-1}$. Let $\Delta$ be the simplicial complex whose faces are the supports of FY monomials. As a divisor of an FY-monomial is an FY-monomial, this is a simplicial complex, and by construction every monomial in $\cF_k$ is a $\Delta$-monomial. Every face of $\Delta$ corresponds to a chain $F_1\subsetneq\cdots\subsetneq F_\ell$ of flats. 

Observe that as $k\leq r-2=\rk(H_i)-1$, $x_{H_i}^k$ belongs to $\cF_k$ for every $i$, so for every $H_i$ there is a corresponding vertex $v_i$ of $\Delta$. We will apply \Cref{lem:separated-shadow} to this complex and the vertices $v_i$. 

First, we claim that the neighborhoods $\Nbd_\Delta [v_i]$ are pairwise disjoint. Suppose for a contradiction that there is a vertex $v$ in $\Nbd_\Delta [v_i] \cap \Nbd_\Delta [v_j]$. This corresponds to a flat $F$ that is comparable to both $H_i$ and $H_j$, with $i \neq j$. 

As $v$ is a vertex of $\Delta$, the variable $x_F$ occurs with positive exponent in an FY-monomial, and hence $\rk(F)\geq2$. Distinct hyperplanes are incomparable, so $F$ is distinct from $H_i$ and $H_j$. Furthermore, as $\rk (E) - \rk (H_i) = 1$, the rank condition for FY monomials does not allow $F = E$, so $F$ is a proper subset of both $H_i$ and $H_j$. Thus $F \subseteq H_i \cap H_j$, which implies $\rk (H_i \cap H_j) \geq 2$, contrary to the assumption $\rk (H_i \cap H_j) \leq 1$ for $i \neq j$.

Now, we can apply \Cref{lem:separated-shadow} to get
\[
 |\partial\cF_k|-\partial_k(|\cF_k|)
 \geq p-\partial_k(p).
\]
As $|\partial\cF_k| \leq|\cF_{k-1}|=a_{k-1}(M)$ and $|\cF_k|=a_k(M)$, we get
\[
 a_{k-1}(M)-\partial_k\bigl(a_k(M)\bigr)
 \geq p-\partial_k(p).
\]
\end{proof}

\section*{Acknowledgements}
The author thanks his advisor Karim Adiprasito for many valuable conversations, comments, and for introducing him to this question. The author is supported by the Oskar Huttunen Foundation and partially by Horizon Europe ERC Grant number: 101045750 / Project acronym: HodgeGeoComb.

\section*{Declaration of AI usage}
An LLM was used to simplify the proofs of \Cref{lem:separated-shadow,mainbound} and to find the further application of \Cref{lem:separated-shadow} to matroid Chow rings. It was also used for proofreading, but not for generating any of the text. The main ideas, results, and proofs are due to the author.
\bibliographystyle{amsplain}
\bibliography{references}

\end{document}